\documentclass[12pt]{article}
\usepackage{amsmath,amsfonts,amssymb,amsthm,bbm,mathtools}
\usepackage{enumerate,booktabs,multirow,hyperref}
\usepackage{graphicx,float,tikz}
\usepackage[linesnumbered,ruled,lined]{algorithm2e}
\usepackage[active]{srcltx}
\usepackage[utf8]{inputenc}

\DeclareMathOperator{\argmin}{argmin}
\DeclareMathOperator{\aff}{aff}

\begin{document}

\newtheorem{oberklasse}{OberKlasse}
\newtheorem{definition}[oberklasse]{Definition}
\newtheorem{lemma}[oberklasse]{Lemma}
\newtheorem{proposition}[oberklasse]{Proposition}
\newtheorem{theorem}[oberklasse]{Theorem}
\newtheorem{corollary}[oberklasse]{Corollary}
\newtheorem{remark}[oberklasse]{Remark}
\newtheorem{example}[oberklasse]{Example}

\newcommand{\R}{\mathbbm{R}}
\newcommand{\N}{\mathbbm{N}}
\newcommand{\Z}{\mathbbm{Z}}
\newcommand{\C}{\mathbbm{C}}
\newcommand{\mc}{\mathcal}

\renewcommand{\odot}
        {\mathop{\bigcirc}}

\long\def\mynote#1{%
        \leavevmode\unskip\raisebox{-3.5pt}{\rlap{$\scriptstyle\diamond$}}%
        \marginpar{\raggedright\hbadness=10000
        \def\baselinestretch{0.8}\tiny
        \it #1\par}}

\title{Generalized Gearhart-Koshy acceleration
is a Krylov subspace method}
\author{Markus Hegland and Janosch Rieger}
\date{\today}
\maketitle

\begin{abstract}
The Kaczmarz method is a row-action method for solving consistent non-square
linear systems, and
Gearhart-Koshy acceleration is a line-search that minimizes the Euclidean
norm of the error along a ray in the direction of a Kaczmarz step.
Recently one of the authors generalized this procedure
to a search for the point with minimal Euclidean error norm within a sequence
of nested affine subspaces.

In this paper, we demonstrate that this generalization can be interpreted
as a Krylov subspace method for a square linear system,
which is equivalent with the original system to be solved.
In exact arithmetic the method cannot break down prematurely, and it makes
progress in every step.
We also present a mathematically equivalent reformulation of the algorithm
in terms of the Gram-Schmidt orthogonalization procedure, and
we illustrate the convergence behavior of the new method
with numerical experiments.
\end{abstract}



\section{Introduction}

The Kaczmarz method is a fixed point iteration for solving consistent non-square
linear systems $Ax=b$, see e.g.\ \cite{Elfving}, \cite{Kaczmarz} and
\cite{Tanabe}, with applications in inverse problems, and, in particular,
in the context of medical imaging.
It is a row-action method \cite{Censor}, which means that it carries out
substeps that use one row of the matrix $A$ at a time.
Recently, the randomized Kaczmarz method \cite{Gower,Strohmer},
which selects the next row to be processed at random, has received a lot
of attention, and different ways of speeding up the basic algorithm have
been explored \cite{Necoara,Needell,Steinerberger}.

\medskip

The Gearhart-Koshy line search aims to accelerate the deterministic Kaczmarz
method \cite{Bauschke,Gearhart,Tam}.
Given a point $x$, it determines the closest point to the solution set of the
linear system $Ax=b$ on the ray that emanates from $x$ and passes through the
point $P(x)$ determined by carrying out one step of the Kaczmarz method from $x$.
One of the authors developed this technique further \cite{Rieger} and derived
an iteration that, in each step, computes the best Euclidean norm approximation
to the solution set of the linear system $Ax=b$ within the affine subspace spanned
by the previous iterates $x_0$ to $x_k$ and the point $P(x_k)$ obtained by
applying the Kaczmarz method to $x_k$.
This optimality indicates that the method described in \cite{Rieger} belongs
to the class of Krylov subspace methods \cite{Liesen, Meurant}, which generate
iterates that minimize a quantity such as a norm of the error or of the residual
over the respective Krylov subspaces, but it is not obvious why this would be
the case.

\medskip

The main purpose of this paper is to show that the method from \cite{Rieger},
when applied to the linear system $Ax=b$, can indeed be interpreted as a Krylov
subspace method for a different linear system $Cx=g$ with a square, but,
in general, non-symmetric matrix $C$, and to develop a mathematically equivalent
numerical method that is based on Gram-Schmidt orthogonalization.
In addition, we slightly generalize the method from \cite{Rieger} to a block
Kaczmarz method, which makes more efficient use of parallel processing
capabilities, and we show that this method makes progress in every single step
and cannot break down prematurely in exact arithmetic.

Finally, we compare the new algorithm with Craig's method, a Krylov subspace method
that also minimizes the Euclidean norm of the error over the Krylov subspaces
it generates, in a number of numerical experiments.
Our computations suggest that the new algorithm cannot compete with Craig's
method on systems with low condition numbers, but that it may perform better
on problems with high condition numbers.

\section{The block Kaczmarz method}

Let $A\in\R^{m\times n}$ and $b\in\R^m$.
Throughout this paper, we assume that the linear system
\begin{equation}\label{to:solve}
Ax=b
\end{equation}
possesses at least one solution.
We consider a partitioning of $A$ and $b$ into matrices
$A_j\in\R^{m_j\times n}$ and vectors $b_j\in\R^{m_j}$ with $j\in\{1,\ldots,p\}$
such that
\[
A^T=(A_1^T,\ldots,A_p^T)
\quad\text{and}\quad
b^T=(b_1^T,\ldots,b_p^T).
\]

\begin{definition}\label{block:Kaczmarz}
A block Kaczmarz iteration for problem \eqref{to:solve} is the fixed point iteration
given by the affine operator
\[P=P_p\circ\ldots\circ P_1\]
with $P_j:\R^n\to\R^n$ of the form
\begin{equation}\label{short:block:form}
P_j(x):=x+A_j^\dagger(b_j-A_jx),\quad
j\in\{1,\ldots,p\}.
\end{equation}
\end{definition}

For an interpretation of the mappings $P_j$ see Lemma \ref{proj:lemma}.
The behavior of the fixed point iteration induced by $P$ is discussed
in Theorem \ref{well:known}.

\begin{remark}\label{reduce:to:Kaczmarz}
When $p=m$ and $m_j=1$ for $j\in\{1,\ldots,p\}$,
i.e.\ when the blocks $A_1,\ldots,A_m$ are just the rows $a_1^T,\ldots,a_m^T$
of the matrix $A$, then the block Kaczmarz method reduces to the usual Kaczmarz
method, which is the composition
$P=P_m\circ\ldots\circ P_1$
of the projections
\[P_j:\R^n\to\R^n,\quad
P_j(x):=(I-\frac{a_ja_j^T}{\|a_j\|^2})x
+\frac{b_j}{\|a_j\|^2}a_j,\quad
j\in\{1,\ldots,m\}.\]
\end{remark}

\medskip

The following lemma details the geometry of a single block Kaczmarz step
and its effect on the square error.
Part b) allows to precompute and store the
pseudoinverses of the small matrices $A_jA_j^T$ instead of the pseudoinverses
of the matrices $A_j$ in Algorithm \ref{Block:kaczmarz:algorithm}.
Part d) is essential for computing the coefficients
in the Gearhart-Koshy acceleration step.
Note that the entries of the vector $w(x)$ are simply the norms
of the steps $d_j$ in Algorithm \ref{Block:kaczmarz:algorithm}.

\begin{lemma}\label{proj:lemma}
The following statements hold.
\begin{itemize}
\item [a)] The operator $P_j$ from equation \eqref{short:block:form} is the $\ell_2$ projector
onto the affine subspace $H_j:=\{z\in\R^n:A_jz=b_j\}$.
\item [b)] We can represent
$P_j(x)=x+A_j^T(A_jA_j^T)^\dagger(b_j-A_jx)$.
\item [c)] We have $\|x-z\|^2=\|x-P_j(x)\|^2+\|P_j(x)-z\|^2$ for all $z\in H_j$.
\item [d)] For every $x\in\R^n$ and every $x_*\in\R^n$ with $Ax_*=b$, we have
\[\|P(x)-x_*\|^2=\|x-x_*\|^2-\|w(x)\|^2,\]
where $w(x)\in\R^m$ is the vector given by
\[w_j(x)
:=\|A_j^T(A_jA_j^T)^\dagger(b_j-A_j(P_{j-1}\circ\ldots\circ P_1(x)))\|
\quad\forall\,j\in\{1,\ldots,p\}.\]
\end{itemize}
\end{lemma}

\begin{algorithm}\label{Block:kaczmarz:algorithm}
\caption{Block-Kaczmarz method}
\KwIn{
    $A_j\in\R^{m_j\times n}$ and
    $b_j\in\R^{m_j}$ for
    $j\in\{1,\ldots,p\}$,
    $y_0\in\R^n$
}
\KwOut{
\begin{minipage}[t]{10cm}
$y_p=P(y_0)$ with $P$ as in Definition \ref{block:Kaczmarz}
and Lemma \ref{proj:lemma}b)\\
$\omega=\|w(y_0)\|^2$ with $w(y_0)$ as in Lemma \ref{proj:lemma}d)
\end{minipage}
}
$\omega=0$\;
\For{$j=1$ \KwTo $p$}{
  $d_j=A_j^T(A_jA_j^T)^\dagger(b_j-A_jy_{j-1})$\;
  $y_j=y_{j-1}+d_j$\;
  $\omega=\omega+\|d_j\|^2$\;
}
\Return $y_p$, $\omega$
\end{algorithm}

\begin{proof}
a) Using that $A_jA_j^\dagger A_j=A_j$ \cite[(2.2.8)]{Bjorck},
that $A_jA_j^\dagger$ is the $\ell_2$ projector onto $\mc{R}(A_j)$
\cite[(2.2.6)]{Bjorck}, and that $b_j\in\mc{R}(A_j)$ by consistency
of $Ax=b$, we obtain
\[
A_jP_j(x)
=A_jx+A_jA_j^\dagger b_j- A_jA_j^\dagger A_jx
=A_jA_j^\dagger b_j
=b_j,
\]
so $P_j(x)\in H_j$.
For any $z\in H_j$, we thus have $P_j(x)-z\in\mc{N}(A_j)$, and since
\[x-P_j(x)=-A_j^\dagger(b_j-A_jx)\in\mc{R}(A_j^\dagger)=\mc{R}(A_j^T)=
\mc{N}(A_j)^\perp,\]
where $\mc{R}(A_j^\dagger)=\mc{R}(A_j^T)$ follows from
\cite[(2.2.1)]{Bjorck} and \cite[(2.2.3)]{Bjorck},
we find
\begin{equation}\label{ortho}
\langle x-P_j(x),P_j(x)-z\rangle=0,
\end{equation}
which implies that $P_j$ is the $\ell_2$ projector onto $H_j$.

b) Let $x_*\in\R^n$ satisfy $Ax_*=b$.
Using $A^\dagger A=(A^\dagger A)^T$ \cite[(2.2.9)]{Bjorck} and
$(A^\dagger)^T=(A^T)^\dagger$ \cite[Theorem 2.2.2.]{Bjorck},
we see that
\[A_j^\dagger A_j
=(A_j^\dagger A_j)^T
=(A_j)^T(A_j^\dagger)^T
=A_j^T(A_j^T)^\dagger,\]
and by \cite[(2.2.7)]{Bjorck} both $A_j^\dagger A_j$ and $A_j^T(A_j^T)^\dagger$
are representations of the (idempotent) $\ell_2$ orthogonal projector onto
$\mc{R}(A_j^T)$.
Further using $(A_j^T)^\dagger A_j^\dagger=(A_jA_j^T)^\dagger$
\cite[Theorem 2.2.3]{Bjorck}, we compute
\begin{align*}
P_j(x)&=x+A_j^\dagger(b_j-A_jx)
=x+A_j^\dagger A_j(x_*-x)
=x+A_j^T(A_j^T)^\dagger A_j^\dagger A_j(x_*-x)\\
&=x+A_j^T(A_jA_j^T)^\dagger A_j(x_*-x)
=x+A_j^T(A_jA_j^T)^\dagger(b_j-A_jx).
\end{align*}

c) is implied by equation \eqref{ortho}.

d) Applying c) with $z=x_*$ gives
\[\|P_j(x)-x_*\|^2=\|x-x_*\|^2-\|P_j(x)-x\|^2,\]
and the desired statement follows by induction.
\end{proof}

The following theorem summarizes geometry and convergence of the
block Kaczmarz method.
Statements d) and e) are part of the reason why the Krylov subspace method
we will develop makes progress in every single step.

\begin{theorem} \label{well:known}
For every $x\in\R^n$, the following statements hold:
\begin{itemize}
\item [a)] There exists a unique $x^*\in\R^n$
with $Ax^*=b$ and $x^*\in x+\mc{R}(A^T)$.
\item [b)] We have $P^k(x)\in x+\mc{R}(A^T)$ for all $k\in\N$.
\item [c)] We have $\|P(x)-P(z)\|\le\|x-z\|$ for all $z\in\R^n$.
\item [d)] We either have $P(x)=x$ or $\|P(x)-x^*\|<\|x-x^*\|$.
\item [e)] We have $Ax=b$ if and only if $P(x)=x$.
\item [f)] We have $\lim_{k\to\infty}P^k(x)=x^*$.
\end{itemize}
\end{theorem}

\begin{proof}
a) By assumption there exists $x_*\in\R^n$ such that $Ax_*=b$.
Since $\R^n=\mc{N}(A)\oplus\mc{R}(A^T)$, there exist
$x^N,x_*^N\in\mc{N}(A)$ and $x^R,x_*^R\in\mc{R}(A)$
with $x=x^N+x^R$ and $x_*=x_*^N+x_*^R$.
Then the vector $x^*:=x_*^R+x^N$ satisfies
\[Ax^*=A(x_*^R+x^N)=A(x_*)+A(-x_*^N+x^N)=b\]
and $x^*-x=x_*^R-x^R\in\mc{R}(A^T)$.
If there are $x_*,x_{**}\in x+\mc{R}(A^T)$ with $Ax_*=b=Ax_{**}$,
then
\[x_*-x_{**}=(x_*-x)-(x_{**}-x)\in\mc{N}(A)\cap\mc{R}(A^T)=\{0\}.\]

\medskip

Statement b) follows from Definition \ref{block:Kaczmarz}
because $\mc{R}(A_j^\dagger)=\mc{R}(A_j^T)\subset\mc{R}(A^T)$.
Statement c) holds, because $P$ is a composition of $\ell_2$ projectors.

\medskip

d)
If $P(x)=x$, then $\|P(x)-x^*\|^2\not<\|x-x^*\|^2$.
If $P(x)\neq x$, then there exists a smallest index $i\in\{1,\ldots,p\}$
with $P_i(x)\neq x$.
In particular, we have
\begin{equation}\label{alle:weg}
\|P(x)-x^*\|^2 =\|P_p\circ P_{p-1}\circ\ldots\circ P_i(x)-x^*\|^2.
\end{equation}
Now for every $j\in\{i,\ldots,p-1\}$, Lemma \ref{proj:lemma}c) with $z=x^*$
yields
\begin{align*}
&\|P_{j+1}\circ P_j\circ\ldots\circ P_i(x)-x^*\|^2\\
&=\|P_j\circ\ldots\circ P_{i}(x)-x^*\|^2
 -\|P_{j+1}\circ\ldots\circ P_i(x)-P_j\circ\ldots\circ P_{i}(x)\|^2\\
&\le\|P_j\circ\ldots\circ P_{i}(x)-x^*\|^2.
\end{align*}
Applying the above argument $p-i$ times to \eqref{alle:weg}
and then Lemma \ref{proj:lemma}c) once more yields
\[\|P(x)-x^*\|^2
\le\|P_{i}(x)-x^*\|^2
=\|x-x^*\|^2-\|P_i(x)-x\|^2
<\|x-x^*\|^2.\]

e)
If $Ax=b$ holds, then \eqref{short:block:form} gives $P_j(x)=x$
for all $j\in\{1,\ldots,p\}$, and hence $P(x)=x$.
If $Ax\neq b$ holds, then there is a smallest index $i\in\{1,\ldots,p\}$
with $0\neq b_i-A_ix$.
Since $b_i-A_ix\in\mc{R}(A_i)$ and
$\mc{N}(A_i^\dagger)=\mc{N}(A_i^T)=\mc{R}(A_i)^\perp$,
we have $P_i(x)=x+A_i^\dagger(b_j-A_jx)\neq x$,
and the same argument as in part d) yields
$\|P(x)-x^*\|^2<\|x-x^*\|^2$ and thus $P(x)\neq x$.

\medskip

f) We will show in Lemma \ref{TC} that we can write
$P(x)=Tx+g$ with a matrix $T\in\R^{n\times n}$ and a vector $g\in\R^n$,
and that relative to the 
decomposition $\R^n=\mc{N}(A)\oplus\mc{R}(A^T)$
we can represent $T=I\oplus T_2$ with $\|T_2\|<1$.
It follows with part b) that $P|_{x+\mc{R}(A^T)}:x+\mc{R}(A^T)\to x+\mc{R}(A^T)$
is a contraction, and the desired statement follows from the contraction
mapping principle.
\end{proof}

Throughout the rest of the paper, we will fix an initial point $x_0\in\R^n$
and the corresponding unique solution $x^*\in\R^n$ of the equation $Ax=b$
as in Theorem \ref{well:known}a).

\section{Generalized Gearhart-Koshy acceleration}

In this section, we characterize Algorithm \ref{kaczmarz:krylov}, which is a
slight generalization of Algorithm 3 from \cite{Rieger}, to the block Kaczmarz
setting.
We follow the approach from \cite{Rieger} very closely.

\begin{algorithm}\label{kaczmarz:krylov}
\caption{Generalized Gearhart-Koshy acceleration}
\KwIn{
    $A_j\in\R^{m_j\times n}$ and
    $b_j\in\R^{m_j}$ for
    $j\in\{1,\ldots,p\}$,
    $x_0\in\R^n$
}
\KwOut{
    solution $x_k\in\R^n$ of the linear system $Ax=b$
}
\For{$k=0$ \KwTo $\infty$}{
    $(P(x_k),\omega_k)
    =\text{[Algorithm \ref{Block:kaczmarz:algorithm}]}
    (A_1,\ldots,A_p,b_1,\ldots,b_p;x_k)$\;
    $r_k=P(x_k)-x_k$\;
    \uIf{$\|r_k\|^2=0$}{\Return $x_k$\;}
    $\gamma_k=(\omega_k+\|r_k\|^2)/2$\;
    $M_k=(x_0-x_k,\ldots,x_{k-1}-x_k,r_k)$\;
    solve $M_k^TM_ks_k=\gamma_k e_{k+1}$ for $s_k$\;
    $x_{k+1}=x_k+M_ks_k$\;
}
\end{algorithm}

The following lemma converts readily available quantities into
information on the location of the exact solution.

\begin{lemma} \label{improvement}
Let $x_0\in\R^n$, and let $x^*\in\R^n$ be the unique solution of the linear
system $Ax=b$ with $x^*\in x_0+\mc{R}(A^T)$.
Then we have
\begin{equation}\label{imp:2}
\|w(x)\|^2+\|P(x)-x\|^2=2\langle x^*-x,P(x)-x\rangle\quad\forall\,
x\in x_0+\mc{R}(A^T).
\end{equation}
\end{lemma}

\begin{proof}
From Lemma \ref{proj:lemma}d) and the polarization identity, it follows that
\begin{align*}
\|w(x)\|^2+\|P(x)-x\|^2
&=\|x-x^*\|^2-\|P(x)-x^*\|^2+\|x-P(x)\|^2\\
&=2\langle x^*-x,P(x)-x\rangle.
\end{align*}
\end{proof}

The following lemma suggests the stopping criterion of Algorithm
\ref{kaczmarz:krylov}.

\begin{lemma} \label{no:stop}
Let $x_0\in\R^n$, let $x^*\in\R^n$ be the unique solution of the linear system
$Ax=b$ with $x^*\in x_0+\mc{R}(A^T)$, and let $x_1,\ldots,x_k\in x_0+\mc{R}(A^T)$ be points with
\begin{equation}\label{last:opt}
x_k=\argmin_{\xi\in\aff(x_0,x_1,\ldots,x_k)}\|\xi-x^*\|^2.
\end{equation}
If $P(x_k)\in\aff(x_0,x_1,\ldots,x_k)$, then we have
$P(x_k)=x_k$ and $x_k=x^*$.
\end{lemma}

\begin{proof}
If $P(x_k)\in\aff(x_1,\ldots,x_k)$, then Theorem \ref{well:known}d)
and \eqref{last:opt} yield that $P(x_k)=x_k$.
Theorem \ref{well:known}e) yields $Ax_k=b$, which in view of Theorem
\ref{well:known}a) implies $x_k=x^*$.
\end{proof}

Applying the following result repeatedly to a given initial point $x_0\in\R^n$
gives rise to Algorithm \ref{kaczmarz:krylov}.

\begin{theorem} \label{affine:search:thm}
Let $x_0\in\R^n$, let $x^*\in\R^n$ be the unique solution
of the linear system $Ax=b$ with $x^*\in x_0+\mc{R}(A^T)$.
Let $x_1,\ldots,x_k\in x_0+\mc{R}(A^T)$ be points such that
the tuple $(x_0,x_1,\ldots,x_k)$ is affinely independent
and \eqref{last:opt} holds.
If $P(x_k)\neq x_k$, then the following statements hold.
\begin{itemize}
\item [a)] The matrix
$M:=(x_0-x_k,\ldots,x_{k-1}-x_k,P(x_k)-x_k)\in\R^{n\times(k+1)}$
has full rank.
\item [b)] The minimizer
$s^*:=\argmin_{s\in\R^{k+1}}\|x_k+Ms-x^*\|^2$
is the unique solution of the linear system
\[M^TMs=\gamma e_{k+1},\]
where $\gamma:=(\|w(x_k)\|^2+\|P(x_k)-x_k\|^2)/2$
and $e_{k+1}\in\R^{k+1}$ denotes the $(k+1)$-th standard basis vector.
\item [c)] The point
\begin{equation}\label{next:point}
x_{k+1}:=x_k+Ms^*
\end{equation}
satisfies $x_{k+1}\in x_0+\mc{R}(A^T)$,
the tuple $(x_0,\ldots,x_k,x_{k+1})$ is affinely independent,
we have $\|x_{k+1}-x^*\|<\|x_k-x^*\|$,
and \eqref{last:opt} holds with $k+1$ in lieu of $k$.
\end{itemize}
\end{theorem}

\begin{proof}
a) Since \eqref{last:opt} and $P(x_k)\neq x_k$ hold, we have
$P(x_k)\notin\aff(x_0,x_1,\ldots,x_k)$ by Lemma \ref{no:stop}.
Since $(x_0,x_1,\ldots,x_k)$ is affinely
independent, this implies that the tuple $(x_0,x_1,\ldots,x_k,P(x_k))$
is affinely independent.
Hence the columns of $M$ are linearly independent.

\medskip

b) In particular, the matrix $M^TM$ is positive definite, and the quadratic
function $g:\R^{k+1}\to\R$ given by $g(s):=\|x_k+Ms-x^*\|^2$
is strictly convex with derivatives
\begin{align*}
\frac{dg}{ds_j}(s)
&=2\langle x_k+Ms-x^*,x_{j-1}-x_k\rangle,
\quad j=1,\ldots,k,\\
\frac{dg}{ds_{k+1}}(s)
&=2\langle x_k+Ms-x^*,P(x_k)-x_k\rangle.
\end{align*}
By \eqref{imp:2} and \eqref{last:opt}, the unique minimizer
$s^*$ of $g$ solves the linear equations
\begin{align*}
&\langle x_{j-1}-x_k,Ms\rangle=\langle x^*-x_k,x_{j-1}-x_k\rangle=0,\quad
j=1,\ldots,k,\\
&\langle P(x_k)-x_k,Ms\rangle=\langle x^*-x_k,P(x_k)-x_k\rangle=\gamma,
\end{align*}
which are displayed in matrix form in statement b).

\medskip

c) Since $x_0,\ldots,x_k\in x_0+\mc{R}(A)$, Theorem \ref{well:known}b)
and the definition \eqref{next:point} yield $x_{k+1}\in x_0+\mc{R}(A)$.
Moreover, definition \eqref{next:point} and the minimality of $s^*$ imply
\begin{equation}\label{almost:optimum}
x_{k+1}=\argmin_{\xi\in\aff(x_0,\ldots,x_k,P(x_k))}\|\xi-x^*\|^2.
\end{equation}
Because of Theorem \ref{well:known}d), the assumption $P(x_k)\neq x_k$ yields
\[\|x_{k+1}-x^*\|\le\|P(x_k)-x^*\|<\|x_k-x^*\|.\]
By \eqref{no:stop}, this shows $x_{k+1}\notin\aff(x_0,x_1,\ldots,x_k)$, so
the tuple $(x_0,\ldots,x_k,x_{k+1})$ is affinely independent.
Hence $\aff(x_0,\ldots,x_k,P(x_k))$ and $\aff(x_0,\ldots,x_k,x_{k+1})$
are both $k+1$-dimensional affine subspaces, and by \eqref{next:point},
we have
\begin{align*}
&\aff(x_0,\ldots,x_k,x_{k+1})\subset\aff(x_0,\ldots,x_k,P(x_k)),
\end{align*}
which means that
\begin{equation}\label{same:spaces}
\aff(x_0,\ldots,x_k,x_{k+1})=\aff(x_0,\ldots,x_k,P(x_k)).
\end{equation}
In view of \eqref{almost:optimum} this shows \eqref{no:stop}
with $k+1$ in lieu of $k$.
\end{proof}

Now we conclude that in exact arithmetic, Algorithm \ref{kaczmarz:krylov}
cannot break down, makes progress in every step, and finds a solution
of the linear system in at most $n$ steps.

\begin{corollary}\label{cor:summary}
Let $x_0\in\R^n$, and let $x^*\in\R^n$ be the unique
solution of the linear system $Ax=b$ with $x^*\in x_0+\mc{R}(A^T)$.
Then the following statements about Algorithm \ref{kaczmarz:krylov} hold.
\begin{itemize}
\item [a)] If the algorithm does not terminate in step $k$, then
the linear system in line 8 possesses a unique solution.
\item [b)] All iterates $x_k$ satisfy \eqref{no:stop}, and we have
$\|x_{k+1}-x^*\|<\|x_k-x^*\|$.
\item [c)] The algorithm terminates with $x_k=x^*$ after $k\le n$ iterations.
\end{itemize}
\end{corollary}

\begin{proof}
Applying Theorem \ref{affine:search:thm} recursively shows that all
statements from Theorem \ref{affine:search:thm}, including statements a) and b)
above, are valid for all iterates the algorithm generates.

If the algorithm terminates in step $k$, then $P(x_k)=x_k$,
and by Theorem \ref{well:known}e), this means that $x_k=x^*$.
If the algorithm does not terminate before step $k=n$,
then Theorem \ref{affine:search:thm}c) yields $\aff(x_0,x_1,\ldots,x_n)=\R^n$.
Statement \eqref{no:stop} implies $x_n=x^*$, and
since $P(x^*)=x^*$, the algorithm terminates.
\end{proof}

The paper \cite{Rieger} goes on to develop a mathematically equivalent
version of Algorithm \ref{kaczmarz:krylov} that avoids solving the normal
equations in line 8 by an updating strategy.
Instead, we will show in the following that Algorithm \ref{kaczmarz:krylov}
is a Krylov subspace method and can be equivalently formulated in terms of a
Gram-Schmidt process.

\section{A Krylov subspace method}

In this section, we demonstrate that Algorithm \ref{kaczmarz:krylov}
is a Krylov subspace method for a system of linear equations $Cx=g$.
Both $C\in\R^{n\times n}$ and $g\in\R^n$ depend on the choice of the dimensions $m_1,\ldots,m_p$
of the blocks $A_1,\ldots,A_p$ and $b_1,\ldots,b_p$, and they generate Krylov spaces
\[\mc{K}_k(C,r_0)
:=\mathrm{span}\big(r_0,Cr_0,\ldots,C^{k-1}r_0\big)\]
with $r_0:=g-Cx_0$.
The following statement can be obtained by induction.

\begin{lemma}\label{P:as:T:g}
The block Kaczmarz method from Definition \ref{block:Kaczmarz} can be written
in the form
\begin{equation}\label{P:as:T:g:formula}
P(x)=Tx+g,
\end{equation}
with $T\in\R^{n\times n}$ and $g\in\R^n$ given by
\begin{equation}\label{def:T}
T:=\prod_{k=1}^p(I-A_k^\dagger A_k),\quad
g:=\sum_{k=1}^p\Big[\prod_{j=k+1}^p(I-A_j^\dagger A_j)\Big]A_k^\dagger b_k.
\end{equation}
where $\prod_{i=j}^pE_i:=E_p\cdots E_j$ and $\prod_{i=p+1}^pE_i:=I$
for $E_1,\ldots,E_p\in\R^{n\times n}$.
\end{lemma}

According to Theorem \ref{well:known}e), finding a solution to $Ax=b$ is equivalent
to finding a vector $x\in\R^n$ with
$x=P(x)=Tx+g$.
Denoting $C:=I-T$, this identity is in turn equivalent with solving the system
\begin{equation}\label{equivalent:system}
Cx=g.
\end{equation}
Note that Algorithm \ref{kaczmarz:krylov} augments the affine basis by the residual
\[r_k=P(x_k)-x_k=g-Cx_k\]
for equation \eqref{equivalent:system} in every step.

\medskip

The structure of the matrices $T$ and $C$ reflects the geometry of the block Kaczmarz
mapping stated in Theorem \ref{well:known}.

\begin{lemma}\label{TC}
Given the orthogonal decomposition $\R^n=\mc{N}(A)\oplus\mc{R}(A^T)$,
the matrices $T$ and $C=I-T$ take the form
\[T=T_1\oplus T_2\quad\text{and}\quad C=C_1\oplus C_2,\]
where
$T_1=I$, $\|T_2\|<1$, $C_1=0$ and $\|I-C_2\|<1$.
The last statement implies that $C_2$ is invertible with
$\sigma(C_2)\subset\mathrm{int}(B_1(1))$.
\end{lemma}

\begin{proof}
If $x\in\mc{N}(A)$, then $(I-A_j^\dagger A_j)x=x$ for all $j\in\{1,\ldots,p\}$,
and hence $Tx=x\in\mc{N}(A)$.
If $x\in\mc{R}(A^T)$, then \eqref{def:T} and $\mc{R}(A^\dagger)=\mc{R}(A^T)$
imply $Tx\in\mc{R}(A^T)$.
We indeed have $T=T_1\oplus T_2$, and $T_1=I$.

\medskip

Let $x^*\in\mc{R}(A^T)$ be the unique solution to $Ax=b$ in $\mc{R}(A^T)$,
and let $z\in\mc{R}(A^T)\setminus\{x^*\}$.
By \eqref{P:as:T:g:formula} and parts e) and d) of Theorem \ref{well:known}, we have
\[\|T(z-x^*)\|
=\|P(z)-P(x^*)\|
=\|P(z)-x^*\|
<\|z-x^*\|.\]
Hence $\|Tx\|<\|x\|$ for all $x\in\mc{R}(A^T)\setminus\{0\}$,
which implies that $\|T_1\|<1$.

\medskip

The properties of $C$ follow from the properties of $T$.
\end{proof}

Algorithm \ref{kaczmarz:krylov} applied to the linear system $Ax=b$ is
a Krylov subspace method for the system $Cx=g$, which minimizes
the Euclidean norm error.

\begin{theorem}\label{what:the:original:algorithm:does}
Let $x_0\in\R^n$, and let $x^*\in\R^n$ be the unique
solution of the linear system $Ax=b$ with $x^*\in x_0+\mc{R}(A^T)$.
If Algorithm \ref{kaczmarz:krylov} generates the vectors $x_1,\ldots,x_k\in\R^n$, then
\[x_k=\argmin_{x\in x_0+\mc{K}_k(C,r_0)}\|x-x^*\|^2,\]
where $r_0:=g-Cx_0$.
\end{theorem}

\begin{proof}
We prove the statement
\begin{equation}\label{spaces:same}
\aff(x_0,\ldots,x_j)=x_0+\mc{K}_j(C,r_0)
\end{equation}
by induction for $j\in\{0,\ldots,k\}$.
The statement of the theorem then follows from Corollary \ref{cor:summary}b).

\medskip

Statement \eqref{spaces:same} is trivial for $j=0$.
Assume that \eqref{spaces:same} holds for some $j\in\{0,\ldots,k-1\}$.
Using \eqref{P:as:T:g:formula}, we see that
\begin{align*}
P(x_j)-x_0
&=Tx_j+g-x_0
=g-Cx_j+x_j-x_0\\
&=(g-Cx_0)-C(x_j-x_0)+(x_j-x_0)\\
&\in r_0+C\mc{K}_j(C,r_0)+\mc{K}_j(C,r_0)
\subset\mc{K}_{j+1}(C,r_0).
\end{align*}
Equation \eqref{same:spaces} from the proof of Theorem \ref{affine:search:thm}
holds for the same reasons as before, and hence
\[\aff(x_0,\ldots,x_j,x_{j+1})
=\aff(x_0,\ldots,x_j,P(x_j))
\subset x_0+\mc{K}_{j+1}(C,r_0).\]
Since $\dim\aff(x_0,\ldots,x_j,x_{j+1})=j+1$ and $\mathrm{dim}(\mc{K}_{j+1}(C,r_0))\le j+1$,
statement \eqref{spaces:same} holds with $j+1$ in lieu of $j$.
\end{proof}

All in all, Algorithm \ref{kaczmarz:krylov} is a Krylov subspace method
for system \eqref{equivalent:system}.

\section{Equivalent Gram-Schmidt-based algorithm}
\label{Gram:Schmidt}

In this section we introduce Algorithm \ref{new:algorithm}, which is based on
Gram-Schmidt orthogonalization.
Comparing Theorems \ref{what:the:original:algorithm:does} and
\ref{new:algorithm:works} shows that Algorithms \ref{kaczmarz:krylov} and
\ref{new:algorithm} generate identical sequences.
Hence Corollary \ref{cor:summary} ensures that when exact arithmetic
is used, Algorithm \ref{new:algorithm} makes progress in every step
and terminates with the exact solution after at most $n$ steps.

\begin{algorithm}\label{new:algorithm}
\caption{Block-Kaczmarz minimal error method (BKME)}
\KwIn{
    $A_j\in\R^{m_j\times n}$ and
    $b_j\in\R^{m_j}$ for
    $j\in\{1,\ldots,p\}$,
    $x_0\in\R^n$
}
\KwOut{
    solution $x_k\in\R^n$ of the linear system $Ax=b$
}
\For{$k=0$ \KwTo $\infty$}{
    $(y_k,\omega_k)
    =\text{[Algorithm \ref{Block:kaczmarz:algorithm}]}
    (A_1,\ldots,A_p,b_1,\ldots,b_p;x_k)$\label{alg:P(xk)}\;
    $r_k=y_k-x_k$\label{alg:rk}\;
    \uIf{$\|r_k\|^2=0$\label{alg:stop}} {\Return $x_k$\;}
    $\tilde q_{k+1}=r_k-\sum_{j=1}^k\langle r_k,q_j\rangle q_j$\;
    $q_{k+1}=\tilde q_{k+1}/\|\tilde q_{k+1}\|$\;
    $\mu_{k+1}=(\omega_k+\|r_k\|^2)/(2\|\tilde q_{k+1}\|)$\;
    $x_{k+1}=x_k+\mu_{k+1}q_{k+1}$\;
}
\end{algorithm}

\begin{theorem}\label{new:algorithm:works}
Algorithm \ref{new:algorithm} is well-defined:
When $r_k\neq 0$, then no division by zero occurs in line 7
and $q_{k+1}\neq 0$.
Furthermore, let $x_0\in\R^n$, and let $x^*\in\R^n$ be the unique
solution of the linear system $Ax=b$ with $x^*\in x_0+\mc{R}(A^T)$.
If Algorithm \ref{new:algorithm} generates the iterates
$x_1,\ldots,x_k\in\R^n$, then
\[x_k=\argmin_{x\in x_0+\mc{K}_k(C,r_0)}\|x-x^*\|^2,\]
where $r_0:=g-Cx_0$.
\end{theorem}

\begin{proof}
Recalling equation \eqref{P:as:T:g:formula} and the notation $C=I-T$,
it follows from lines \ref{alg:P(xk)} and \ref{alg:rk}
that
\begin{itemize}
\item [a)] we have $r_j=g-Cx_j$ for $j\in\{0,\ldots,k\}$.
\end{itemize}
In the following, we prove that
\begin{itemize}
\item [b)] we have $\tilde q_j\neq 0$ when $j>0$,
\item [c)] the tuple $(q_1,\ldots,q_j)$ is an orthonormal basis of $\mc{K}_j(C,r_0)$, and
\item [d)] we have
$x_j=\argmin_{x\in x_0+\mc{K}_j(C,r_0)}\|x-x^*\|^2$
\end{itemize}
by induction for $j\in\{0,\ldots,k\}$.
The statements of the theorem then follow.

\medskip

Since $\mc{K}_0(C,r_0)=\{0\}$, there is nothing to show when $j=0$.
Now assume that statements b) to d) hold for some $j\in\{0,\ldots,k-1\}$.

Because of a) and induction hypothesis d), we find
\begin{equation}\label{rk:in}
r_j=g-Cx_j=r_0+C(x_0-x_j)\in\mc{K}_{j+1}(C,r_0).
\end{equation}
Since the algorithm did not terminate in $j$-th iteration,
the stopping criterion in line \ref{alg:stop} yields
$\|P(x_j)-x_j\|^2=\|r_j\|^2\neq 0$.
In view of Theorem \ref{well:known}d) and induction hypothesis d), this implies
\[
\|x_j+r_j-x^*\|^2
=\|P(x_j)-x^*\|^2
<\|x_j-x^*\|^2
=\min_{x\in x_0+\mc{K}_j(C,r_0)}\|x-x^*\|^2.
\]
Hence $x_j+r_j\notin\mc{K}_j(C,r_0)$, and since $x_j\in\mc{K}_j(C,r_0)$,
we have $r_j\notin\mc{K}_j(C,r_0)$.
By \eqref{rk:in} and c), we thus have
\[\tilde q_{j+1}=r_j-\sum_{i=1}^j\langle r_j,q_i\rangle q_i\in\mc{K}_{j+1}(C,r_0)\setminus\{0\},\]
which implies b) with $j+1$ in lieu of $j$,
as well as $q_{j+1}\in\mc{K}_{j+1}(C,r_0)$.
Since $\tilde{q}_{j+1}\neq 0$, a division by zero cannot occur.
Because of c) and since
\begin{equation}\label{by:construction}
\langle q_{j+1},q_\ell\rangle
=\tfrac{1}{\|\tilde q_{j+1}\|}\langle r_j-\sum_{i=1}^j\langle r_j,q_i\rangle q_i,q_\ell\rangle
=0\quad\forall\,\ell\in\{1,\ldots,j\},
\end{equation}
statement c) holds with $j+1$ in lieu of $j$.
We also have
\[x_{j+1}=x_j+\mu_{j+1}q_{j+1}\in x_0+\mc{K}_{j+1}(C,r_0).\]
Parts c) and d) of the induction hypothesis imply that
\begin{equation}\label{orthogonal}
\langle x_j-x^*,q_i\rangle=0\quad\forall\,i\in\{1,\ldots,j\}.
\end{equation}
Combining statements \eqref{by:construction} and \eqref{orthogonal},
we obtain
\begin{equation}\label{first:k}
\langle x_{j+1}-x^*,q_i\rangle
=\langle x_{j+1}-x_j,q_i\rangle+\langle x_j-x^*,q_i\rangle
=\mu_{j+1}\langle q_{j+1},q_i\rangle
=0
\end{equation}
for all $i\in\{1,\ldots,j\}$.
Statement \eqref{orthogonal} also implies that
\begin{equation}\label{shift}
\langle x_j-x^*,r_j\rangle
=\langle x_j-x^*,\tilde{q}_{j+1}+\sum_{i=1}^j\langle r_j,q_i\rangle q_i\rangle
=\langle x_j-x^*,\tilde{q}_{j+1}\rangle.
\end{equation}
Using \eqref{imp:2} and \eqref{shift}, we see that
\[
\mu_{j+1}
=\tfrac{1}{\|\tilde{q}_{j+1}\|}\langle x^*-x_j,r_j\rangle
=\tfrac{1}{\|\tilde{q}_{j+1}\|}\langle x^*-x_j,\tilde{q}_{j+1}\rangle\\
=\langle x^*-x_j,q_{j+1}\rangle,
\]
so we obtain
\begin{equation}\label{last:k}\begin{aligned}
\langle x_{j+1}-x^*,q_{j+1}\rangle
&=\langle x_j+\mu_{j+1}q_{j+1}-x^*,q_{j+1}\rangle\\
&=\langle x_j-x^*,q_{j+1}\rangle
+\mu_{j+1}
= 0.
\end{aligned}\end{equation}
Since $q_1,\ldots,q_{j+1}$ is an orthonormal basis of $\mc{K}_{j+1}(C,r_0)$,
statements \eqref{first:k} and \eqref{last:k} imply
d) with $j+1$ in lieu of $j$.
\end{proof}

We show that there exists a prominent example in which the assumptions
of Theorem \ref{BKME:error:estimate} hold.

\begin{lemma}
Given the orthogonal decomposition $\R^n=\mc{N}(A)\oplus\mc{R}(A^T)$,
the symmetric block Kaczmarz method
\[\tilde P=P_1\circ\ldots\circ P_{p-1}\circ P_p\circ P_{p-1}\circ\ldots\circ P_1\]
can be written in the form $\tilde P(x)=\tilde Tx+\tilde g$
with $\tilde T=I\oplus\tilde T_2$.
The matrix $\tilde C:=I-\tilde T\in\R^{n\times n}$ can be written
in the form $\tilde C=0\oplus\tilde C_2$, and $\tilde C_2$
is symmetric and positive definite.
\end{lemma}

\begin{proof}
The operator $\tilde P$ coincides with the block Kaczmarz method from Definition
\ref{block:Kaczmarz} applied to the data
\[
\tilde A=(A_1^T,\ldots,A_{p-1}^TA_p^TA_{p-1}^T,\ldots,A_1^T)^T
\ \text{and}\
\tilde b=(b_1^T,\ldots,b_{p-1}^Tb_p^Tb_{p-1}^T,\ldots,b_1^T)^T,
\]
so Lemmas \ref{P:as:T:g} and \ref{TC} apply.
Lemma \ref{P:as:T:g} yields a representation
\[\tilde T=(I-A_1^\dagger A_1)\cdots(I-A_{p-1}^\dagger A_{p-1})
(I-A_{p}^\dagger A_{p})(I-A_{p-1}^\dagger A_{p-1})\cdots(I-A_1^\dagger A_1).\]
Since $(A_j^\dagger A_j)^T=A_j^\dagger A_j$ by \cite[(2.2.9)]{Bjorck},
we have $(I-A_j^\dagger A_j)^T=I-A_j^\dagger A_j$ for $j\in\{1,\ldots,p\}$,
and hence $\tilde T^T=\tilde T$.
This implies $\tilde C^T=\tilde C$ and hence $\tilde C_2^T=\tilde C_2$.
Lemma \ref{TC} yields $\sigma(\tilde C_2)\subset\mathrm{int}(B_1(1))$,
and therefore $\tilde C_2$ is positive definite.
\end{proof}

We provide an error estimate for BKME.
Note that, in general, the matrix $C_2$ need not be symmetric.

\begin{theorem}\label{BKME:error:estimate}
Let $x_0\in\R^n$, and let $x^*\in\R^n$ be the unique
solution of the linear system $Ax=b$ with $x^*\in x_0+\mc{R}(A^T)$.
If Algorithm \ref{new:algorithm} generates the iterates
$x_1,\ldots,x_k\in\R^n$ and the matrix $C_2$ is symmetric positive definite,
then we have
\begin{equation}\label{BKME:error}
\|e_{k}\|\le 2\left(\frac{\sqrt{\kappa(C_2)}-1}{\sqrt{\kappa(C_2)}+1}\right)^k\|e_{0}\|,
\end{equation}
where $e_k:=x^*-x_k$ and $\kappa(C_2):=\lambda_{\max}(C_2)/\lambda_{\min}(C_2)$.
\end{theorem}

\begin{proof}
Theorem \ref{new:algorithm:works} can be restated in the form
\begin{equation*}
e_k=\argmin_{e\in e_0+\mathrm{span}(Ce_0, C^2e_0,\ldots,C^ke_0)}\|e\|^2.
\end{equation*}
Let $\mc{P}_k$ be the space of polynomials of degree at most $k$,
and denote
\[\hat{\mc{P}}_k:=\{p\in\mc{P}_k: p(0)=1\}.\]
Then $e_k=p_k(C)e_0$, where
$p_k=\argmin_{p\in\hat{\mc{P}}_k}\|p(C)e_0\|$.
We take an eigenvalue factorization $C_2=V\Lambda V^{-1}$,
and since $e_0\in\mc{R}(A^T)$, Lemma \ref{TC} yields
\begin{align*}
\|e_{k}\|
&=\min_{p\in\hat{\mc{P}}_k}\|p(C)e_0\|
=\min_{p\in\hat{\mc{P}}_k}\|p(C_2)e_0\|\\
&=\min_{p\in\hat{\mc{P}}_k}\|p(V\Lambda V^{-1})e_0\|
=\min_{p\in\hat{\mc{P}}_k}\|Vp(\Lambda)V^{-1}e_0\|
\le\min_{p\in\hat{\mc{P}}_k}\|p(\Lambda)\|\,\|e_0\|.
\end{align*}
Now we can follow verbatim the proof of Theorem 3.1.1 in \cite{Greenbaum}
and find that the convergence rate is bounded by
\[
\frac{\|e_{k}\|}{\|e_0\|}
\le\min_{p\in\hat{\mc{P}}_k}\|p(\Lambda)\|
=\min_{p\in\hat{\mc{P}}_k}\max_{i=1,\ldots,n}|p(\lambda_i)|
\le 2\left(\tfrac{\sqrt{\kappa(C_2)}-1}{\sqrt{\kappa(C_2)}+1}\right)^k,
\]
where the final step is achieved by constructing a $p\in\hat{\mc{P}}_k$
for which the bound holds.
\end{proof}

Up to our knowledge, there is no theoretical result relating
$\kappa(C_2)$ to $\kappa(A)$.
Tables \ref{table:info:A} and \ref{table:data:C}
detail these numbers for the numerical experiments carried out in this paper.

\section{Comparison with Craig's method}

\begin{algorithm}\label{alg:CGNE}
\caption{Craig's method (CGME or CGNE)}
\KwIn{
    $A\in\R^{m\times n}$ and
    $b\in\R^m$,
    $x_0\in\R^n$
}
\KwOut{
    solution $x_k\in\R^n$ of the linear system $Ax=b$
}
$r_0=b-Ax_0$\;
$p_0=A^Tr_0$\;
\For{$k=0$ \KwTo $\infty$}{
    $\alpha_k=\|r_k\|^2/\|p_k\|^2$\;
	$x_{k+1}=x_k+\alpha_kp_k$\;
	$r_{k+1}=r_k-\alpha_kAp_k$\;
	$\beta_k=\|r_{k+1}\|^2/\|r_k\|^2$\;
	$p_{k+1}=A^Tr_{k+1}+\beta_kp_k$\;
}
\end{algorithm}

Craig's method, also called CGME (CG minimal error) or
CGNE (CG on normal equations), is derived by applying CG to the system
\begin{equation*}
AA^Tu=b.
\end{equation*}
Since CG converges on consistent positive semi-definite problems, the method
is well-defined \cite[Section 4.5.3]{Bjorck}.
It is shown in \cite[Section 8.3.2]{Saad} that the transformation $x=A^Tu$
converts this iteration into Algorithm \ref{alg:CGNE}.
It is easy to check that the directions $p_k$ are mutually orthogonal w.r.t.\
the Euclidean inner product, and the minimality of the $AA^T$-norm error
of the CG iterates translates into the Euclidean optimality
\begin{equation*}
x_k=\argmin_{x\in x_0+\mc{K}_k(A^TA,A^Tr_0)}\|x-x^*\|^2
\end{equation*}
for CGME.
This poperty makes CGME the natural benchmark for BKME.
Note that this equation differs from the statement of Theorem
\ref{new:algorithm:works}, because BKME and CGME generate different
Krylov subspaces.

\medskip

Working with the Gram matrix $AA^T$ comes at the cost of squaring the condition number
\cite[Section 8.1]{Saad}.
This manifests in the error estimate
\begin{equation}\label{CGME:error:estimate}
\|x^*-x_k\|\le 2\left(\frac{\kappa(A)-1}{\kappa(A)+1}\right)^k\|x^*-x_0\|
\end{equation}
for CGME, which can be derived from the well-known error estimate for CG.
As a relationship between $\kappa(C_2)$ and $\kappa(A)$ does not seem
to be known, we cannot directly compare the error estimates \eqref{BKME:error}
and \eqref{CGME:error:estimate}.

\medskip

While CGME needs procedures to form matrix-vector products with both $A$
and $A^T$, BKME in its most basic form with $p=m$
as in Remark \ref{reduce:to:Kaczmarz} only needs access to one row of $A$
at a time.
This has been considered a decisive advantage of row-action methods
\cite{Censor}.
From a parallelization standpoint, however, it is often advantageous
to work with a full matrix-vector product, which favors CGME.

\medskip

Bj\"orck observes in \cite[Section 4.5.3]{Bjorck} that while the error of CGME
is monotone decreasing, the residual can fluctuate wildly and may therefore not
be a good basis for a stopping criterion.
Instead, based on a backward error analysis of the least-squares problem
$\min_{x\in\R^n}\|b-Ax\|$, he suggests to terminate CGME for consistent
systems $Ax=b$ when
\begin{equation}\label{Bjorck:stopping:criterion}
\|b-Ax_k\|\le\epsilon(\|A\|\|x_k\|+\|b\|),
\end{equation}
where $\epsilon$ is a small multiple of machine precision.
It is a drawback of this approach that the norm $\|A\|$ has to be either known
or estimated.

\medskip

Both CGME and BKME eventually become unstable due to round-off errors and need
to be restarted.
CGME has a lower complexity per step than BKME because of the Gram-Schmidt process,
which becomes a burden for large iteration numbers.
The significance of this difference depends on the performance of the
methods on a given problem and the restarting policy.

\section{Numerical experiments}

We compare BKME and CGME when applied to three model problems
from spherical Radon transform tomography (spericaltomo), parallel
beam X-ray CT (paralleltomo) and seismic travel time tomography
(seismicwavetomo).
We use the Air Tools II library \cite{Hansen} with default settings
to generate linear systems corresponding to the resolutions
of 32x32 pixels, 64x64 pixels and 128x128 pixels, respectively,
for each problem, choosing the Shepp-Logan phantom for sphericaltomo
and paralleltomo, and the tectonic subduction zone for seismicwavetomo
as objects to be reconstructed.
The significance of the problems and the construction of the linear
systems is detailed in \cite{Hansen}.
As the Kaczmarz method is sensitive to the ordering of the rows,
we shuffle the rows randomly to avoid artifacts in the behavior
of BKME from the specific order in which the rows of the system
are generated.
Table \ref{table:info:A} summarizes key properties of the resulting matrices,
where $\mathrm{nnz}(A)$, $\delta(A)$, $\|A\|$ and $\kappa(A)$ denote
the number of nonzero elements, the density, the spectral norm
and the spectral condition number of $A$.
The final column shows the convergence rate of CGME from
bound \eqref{CGME:error:estimate}.

\medskip

\begin{table}
\footnotesize\centering
\begin{tabular}{lrrrrrrl}
\toprule
& pixels & size(A)& nnz(A) & $\delta(A)$ & $\|A\|$ & $\kappa(A)$
& $\frac{\kappa(A)-1}{\kappa(A)+1}$\\
\midrule
spherical- & 32x32 & 7092x1024 & 204382 & 0.0281 & 3.37 & 23 & 0.912\\
tomo & 64x64 & 14341x4096 & 825366 & 0.0141 & 2.39 & 37 & 0.95\\
& 128x128 & 28590x16384 & 3283898 & 0.0070 & 1.69 & 116 & 0.98\\
\midrule
parallel- & 32x32 & 7330x1024 & 234272 & 0.0312 & 75 & 340 & 0.994\\
tomo & 64x64 & 14686x4096 & 938572 & 0.0156 & 106 & 1010 & 0.998\\
& 128x128 & 29370x16384 & 3754696 & 0.0078 & 149 & 3631 & 0.9995\\
\midrule
seismic- & 32x32 & 2048x1024 & 334022 & 0.1593 & 50 & 7.79E7
& 1-0.25E-7\\ 
wavetomo & 64x64 & 8192x4096 & 5340749 & 0.1592 & 101 & 1.67E8
& 1-0.12E-7\\ 
& 128x128 & 32768x16384 & 85345237 & 0.1590 & 202 & 7.39E8
& 1-0.27E-8\\ 
\bottomrule
\end{tabular}
\caption{Basic information on test matrices $A$ used in numerical experiments.}
\label{table:info:A}
\end{table}

\begin{table}\scriptsize\centering
\begin{tabular}{p{0.5cm}rrrrrr}
\toprule
\multirow{9}{*}{\rotatebox[origin=c]{90}
{sphericaltomo\hspace{14.5ex}}} &
32x32 pixels & s=2 & s=4 & s=8 & s=16 & s=32\\
& $\|C\|$ & 1.10 & 1.10 & 1.10 & 1.10 & 1.10\\
& $\kappa(C)$ & 3.54 & 3.53 & 3.51 & 3.48 & 3.40\\
& $\frac{\sqrt{\kappa(C)}-1}{\sqrt{\kappa(C)}+1}$
& 0.31 & 0.31 & 0.30 & 0.30 & 0.30\\
\cmidrule{2-7}
& 64x64 pixels & s=2 & s=4 & s=8 & s=16 & s=32\\
& $\|C\|$ & 1.23 & 1.23 & 1.23 & 1.23 & 1.23\\
& $\kappa(C)$ & 7.85 & 7.84 & 7.82 & 7.81 & 7.77\\
& $\frac{\sqrt{\kappa(C)}-1}{\sqrt{\kappa(C)}+1}$
& 0.47 & 0.47 & 0.47 & 0.47 & 0.47\\
\cmidrule{2-7}
&128x128 pixels & s=2 & s=4 & s=8 & s=16 & s=32\\
& $\|C\|$ & 1.31 & 1.30 & 1.31 & 1.31 & 1.31\\
& $\kappa(C)$ & 53.7 & 53.7 & 53.7 & 53.6 & 53.6\\
& $\frac{\sqrt{\kappa(C)}-1}{\sqrt{\kappa(C)}+1}$
& 0.76 & 0.76 & 0.76 & 0.76 & 0.76\\
\midrule
\multirow{9}{*}{\rotatebox[origin=c]{90}{paralleltomo\hspace{15ex}}} &
32x32 pixels & s=2 & s=4 & s=8 & s=16 & s=32\\
& $\|C\|$ & 1.10 & 1.10 & 1.10 & 1.10 & 1.09\\
& $\kappa(C)$ & 608 & 605 & 596 & 588 & 566\\
& $\frac{\sqrt{\kappa(C)}-1}{\sqrt{\kappa(C)}+1}$
& 0.92 & 0.92 & 0.92 & 0.92 & 0.92\\
\cmidrule{2-7}
& 64x64 pixels & s=4 & s=8 & s=16 & s=32 & s=64\\
& $\|C\|$ & 1.22 & 1.22 & 1.23 & 1.23 & 1.23\\
& $\kappa(C)$ & 5143 & 5129 & 5098 & 5062 & 4983\\
& $\frac{\sqrt{\kappa(C)}-1}{\sqrt{\kappa(C)}+1}$
& 0.97 & 0.97 & 0.97 & 0.97 & 0.97\\
\cmidrule{2-7}
& 128x128 pixels & s=8 & s=16 & s=32 & s=64 & s=128\\
& $\|C\|$ & 1.31 & 1.31 & 1.31 & 1.32 & 1.32\\
& $\kappa(C)$ & 56318 & 56273 & 56134 & 55938 & 55411\\
& $\frac{\sqrt{\kappa(C)}-1}{\sqrt{\kappa(C)}+1}$
& 0.9916 & 0.9916 & 0.9916 & 0.9916 & 0.9915\\
\midrule
\multirow{9}{*}{\rotatebox[origin=c]{90}{seismicwavetomo\hspace{13ex}}} &
32x32 pixels & s=4 & s=8 & s=16 & s=32 & s=64\\
& $\|C\|$ & 1.08 & 1.10 & 1.12 & 1.13 & 1.17\\
& $\kappa(C)$ & 1.81E13 & 1.78E13 & 1.66E13 & 1.38E13 & 9.77E12\\
& $\frac{\sqrt{\kappa(C)}-1}{\sqrt{\kappa(C)}+1}$
& 1-0.47E-6 & 1-0.47E-6 & 1-0.49E-6 & 1-0.54E-6 & 1-0.64E-6 \\
\cmidrule{2-7}
& 64x64 pixels & s=8 & s=16 & s=32 & s=64 & s=128\\
& $\|C\|$ & 1.06 & 1.07 & 1.08 & 1.10 & 1.11\\
& $\kappa(C)$ & 2.18E13 & 2.01E13 & 1.68E13 & 1.25E13 & 7.90E12\\
& $\frac{\sqrt{\kappa(C)}-1}{\sqrt{\kappa(C)}+1}$
& 1-0.43E-6 & 1-0.45E-6 & 1-0.49E-6 & 1-0.57E-6 & 1-0.71E-6 \\
\cmidrule{2-7}
& 128x128 pixels & s=16 & s=32 & s=64 & s=128 & s=256\\
& $\|C\|$ & 1.05 & 1.06 & 1.07 & 1.07 & 1.07\\
& $\kappa(C)$ & 1.00E14 & 8.79E13 & 6.74E13 & 4.42E13 & 2.25E13\\
& $\frac{\sqrt{\kappa(C)}-1}{\sqrt{\kappa(C)}+1}$
& 1-0.20E-6 & 1-0.21E-6 & 1-0.24E-6 & 1-0.30E-6 & 1-0.42E-6\\
\bottomrule
\end{tabular}
\caption{Basic information on matrices $C$ corresponding to test matrices $A$
from Table \ref{table:info:A} with various block sizes $s$.
In all examples we have $C_2=C$.}
\label{table:data:C}
\end{table}

\begin{figure}\begin{center}
\includegraphics[trim={1.1cm 1cm 1.2cm 0.5cm},clip]{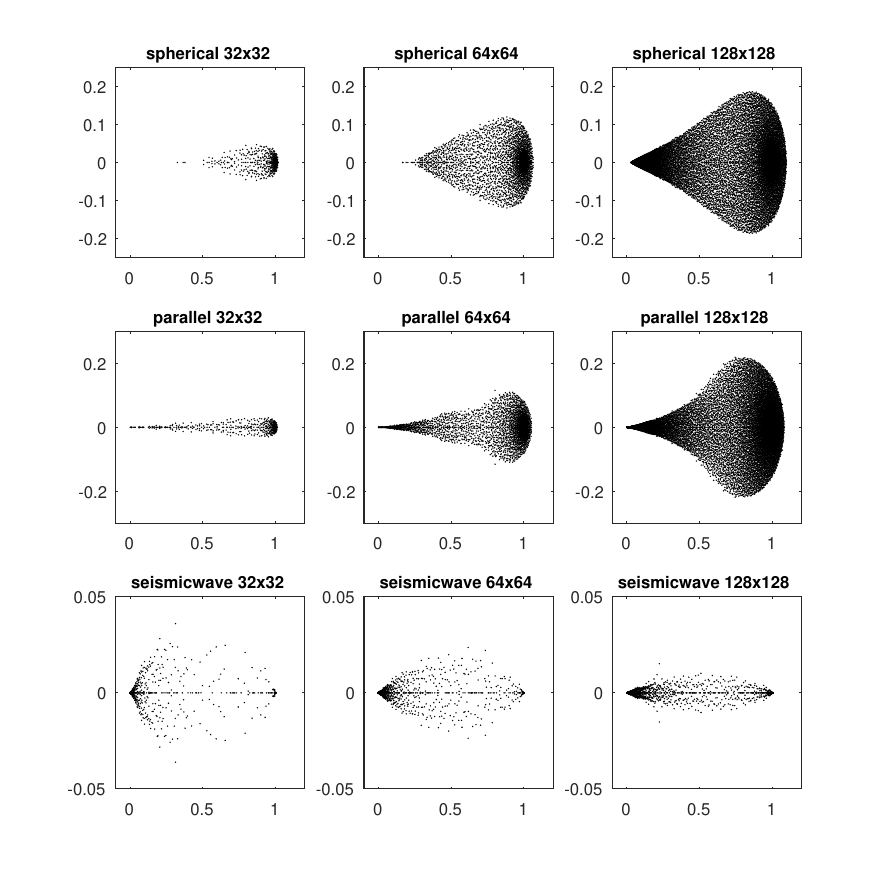}
\caption{Spectrum of matrix $C$ for all test matrices $A$.}
\label{fig:spectra}
\end{center}\end{figure}

\begin{figure}\begin{center}
\includegraphics[trim={0.5cm 0.5cm 1.2cm 0.5cm},clip]{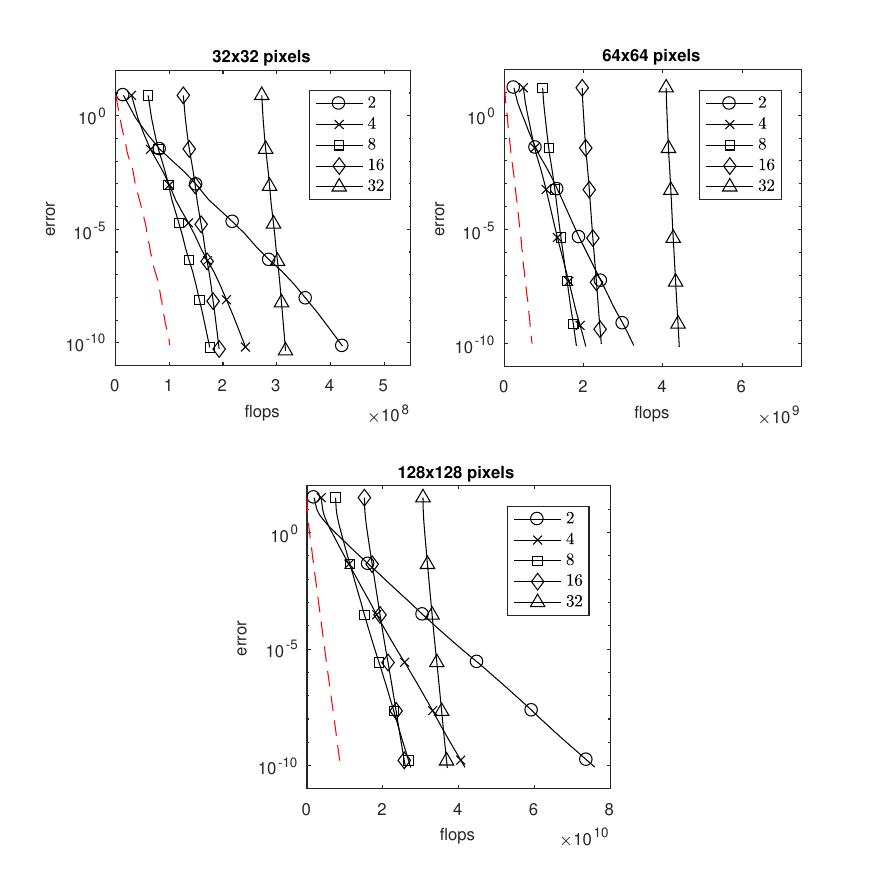}
\caption{Performance of BKME on problem \emph{sphericaltomo} for three different
resolutions and multiple block sizes.
Horizontal displacements of initial points reflect precomputation
of pseudoinverses.
Dashed red line is CGME.}
\label{fig:sphericaltomo}
\end{center}\end{figure}

\begin{figure}\begin{center}
\includegraphics[trim={0.5cm 0.5cm 1.2cm 0.5cm},clip]{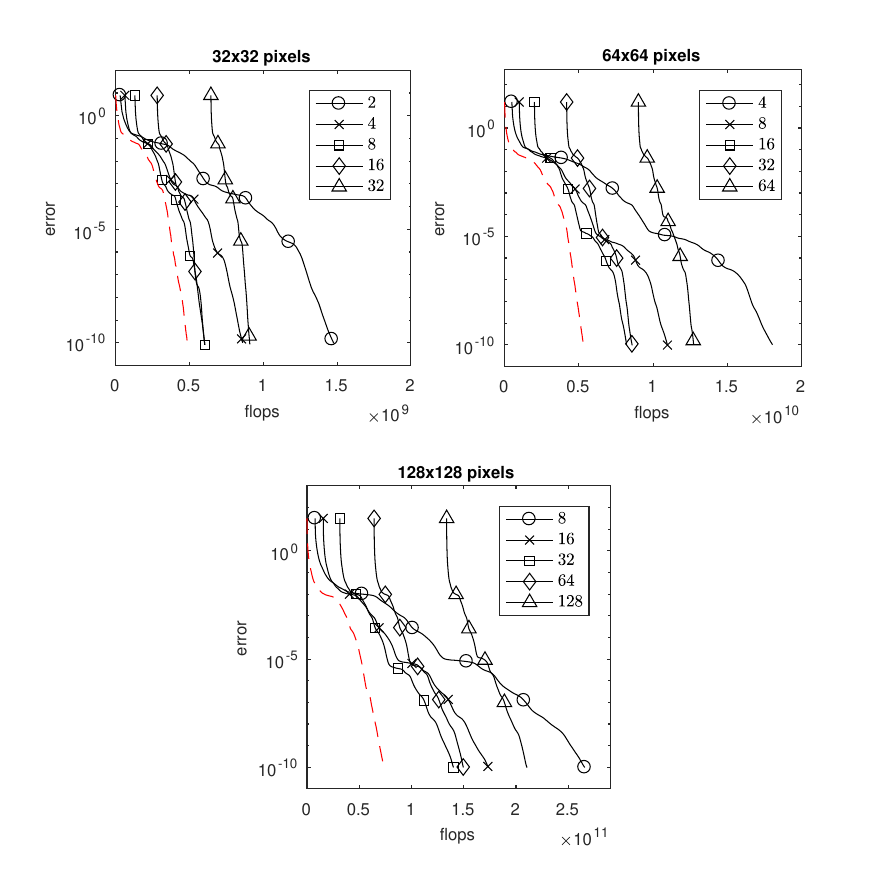}
\caption{Performance of BKME on problem \emph{paralleltomo} for three different
resolutions and multiple block sizes.
Horizontal displacements of initial points reflect precomputation
of pseudoinverses.
Dashed red line is CGME.}
\label{fig:paralleltomo}
\end{center}\end{figure}

\begin{figure}\begin{center}
\includegraphics[trim={0.5cm 0.5cm 1.2cm 0.5cm},clip]{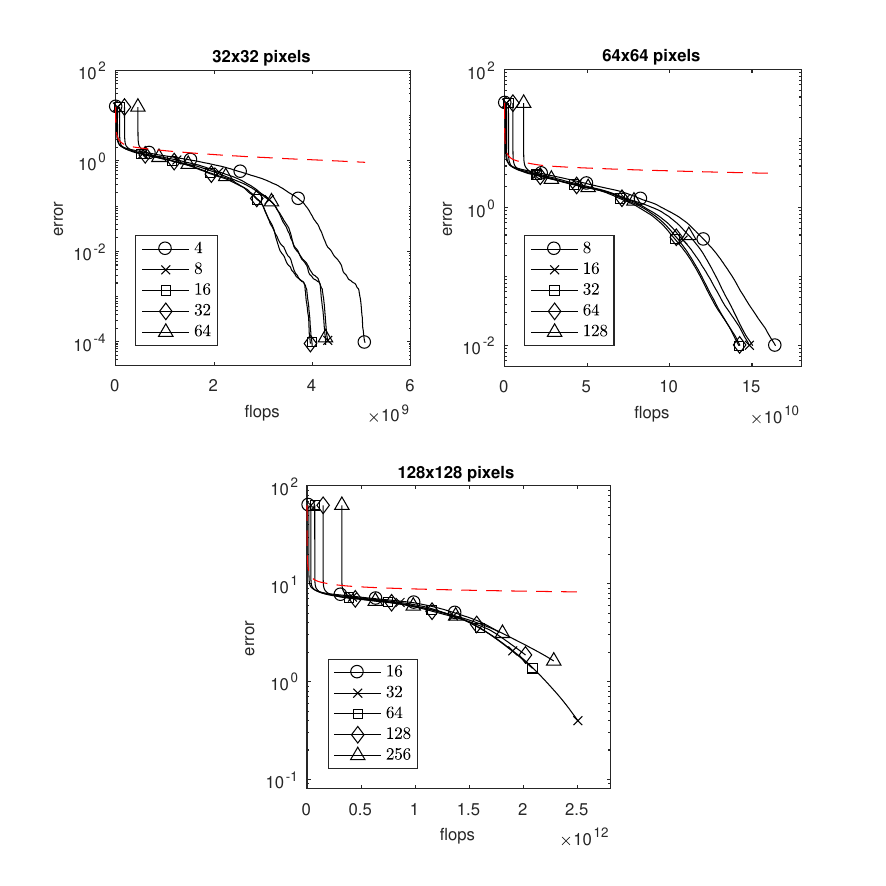}
\caption{Performance of BKME on problem \emph{seismicwavetomo} for
different resolutions and multiple block sizes.
Horizontal displacements of initial points reflect precomputation
of pseudoinverses.
Dashed red line is CGME.}
\label{fig:seismicwavetomo}
\end{center}\end{figure}

As no quantitative results on the relationship between the matrix $A$ and the
corresponding matrix $C$ seem to be known, we compute and examine the
matrix $C$ for every matrix $A$ and every block size we consider.
Table \ref{table:data:C} shows the spectral norm $\|C\|$ and the spectral
condition number $\kappa(C)$ for every such matrix $C$.
In particular, we see that $C$ is invertible, and hence the matrix $C_2$
from Lemma \ref{TC} coincides with $C$ in all examples.
In addition, Table \ref{table:data:C} shows the convergence rate of BKME
that is suggested by estimate \eqref{BKME:error}.
Recall that the bound only holds for symmetric matrices $C$ and may,
therefore, not be satisfied in our examples.
We plot the spectra of the matrices $C$ in Figure \ref{fig:spectra}.
As the block size of BKME does not seem to have a major influence on the
visual appearance of the spectrum, we only plot the spectrum of $C$
for one particular block size for every problem and resolution.

\medskip

Figures \ref{fig:sphericaltomo}, \ref{fig:paralleltomo} and
\ref{fig:seismicwavetomo} show the performance of both BKME and CGME
when applied to all model problems.
It is measured in terms of the error $\|e_k\|=\|x^*-x_k\|$
and the flops carried out to achieve this accuracy, including the
work invested into the precomputation of pseudoinverses and the
Gram-Schmidt process for BKME.
The iterations are stopped when a certain error or a certain number
of flops is reached.

The plots show that the optimal block size for BKME varies with
both problem type and resolution.
They also suggest that CGME may outperform BKME on systems with
a low condition number, and that BKME may outperform CGME on
systems with a higher condition number.

Though error estimate \eqref{BKME:error:estimate} does not apply in the
situation of our experiments,
because it holds for symmetric matrices $C$ only, it is only violated
in a minor way in problem sphericaltomo 128x128, and it is correct for
every other computation we made.
It is, however, quite pessimistic for problem paralleltomo, and
extremely pessimistic for problem seismicwavetomo.

We found that clock time was not a reliable measure for performance,
because it seemed to depend strongly on implementational details
and quirks of the programming language (Julia) used.
However, it seemed clear that larger blocks lead to shorter
runtimes for implementations in Julia.

\medskip

\begin{figure}\begin{center}
\includegraphics[trim={0.5cm 0.5cm 1.2cm 0.5cm},clip]{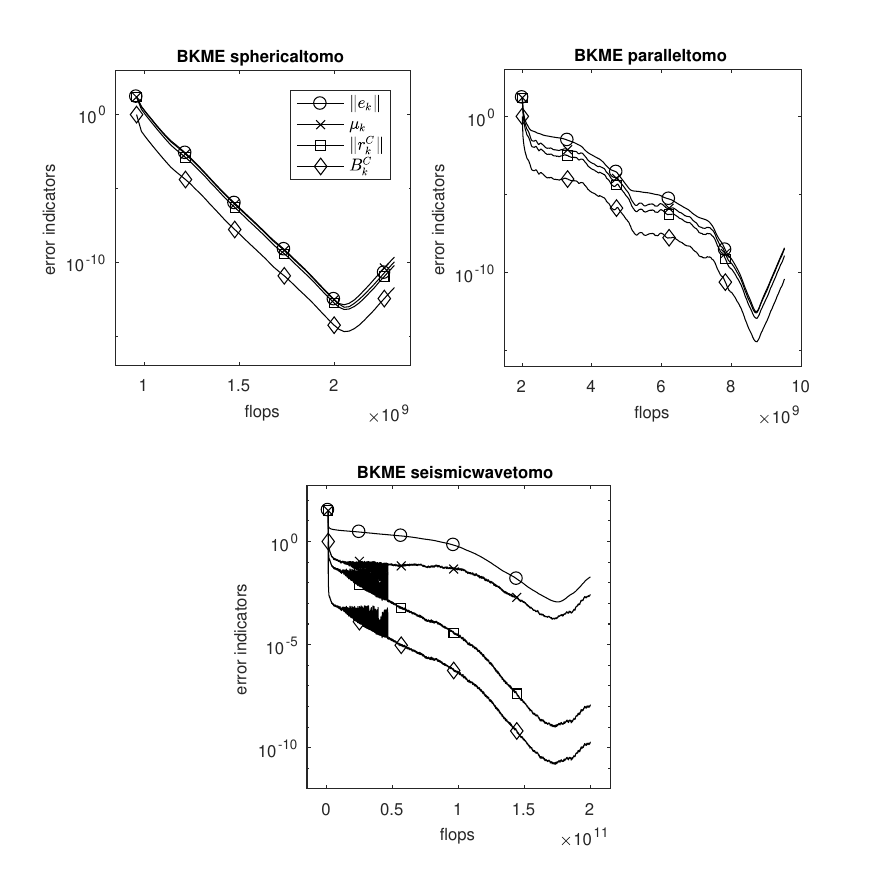}
\caption{Behavior of error indicators for BKME on model problems
with 64x64 pixels.}
\label{fig:kaczmarz}
\end{center}\end{figure}

\begin{figure}\begin{center}
\includegraphics[trim={0.5cm 0.5cm 1.2cm 0.5cm},clip]{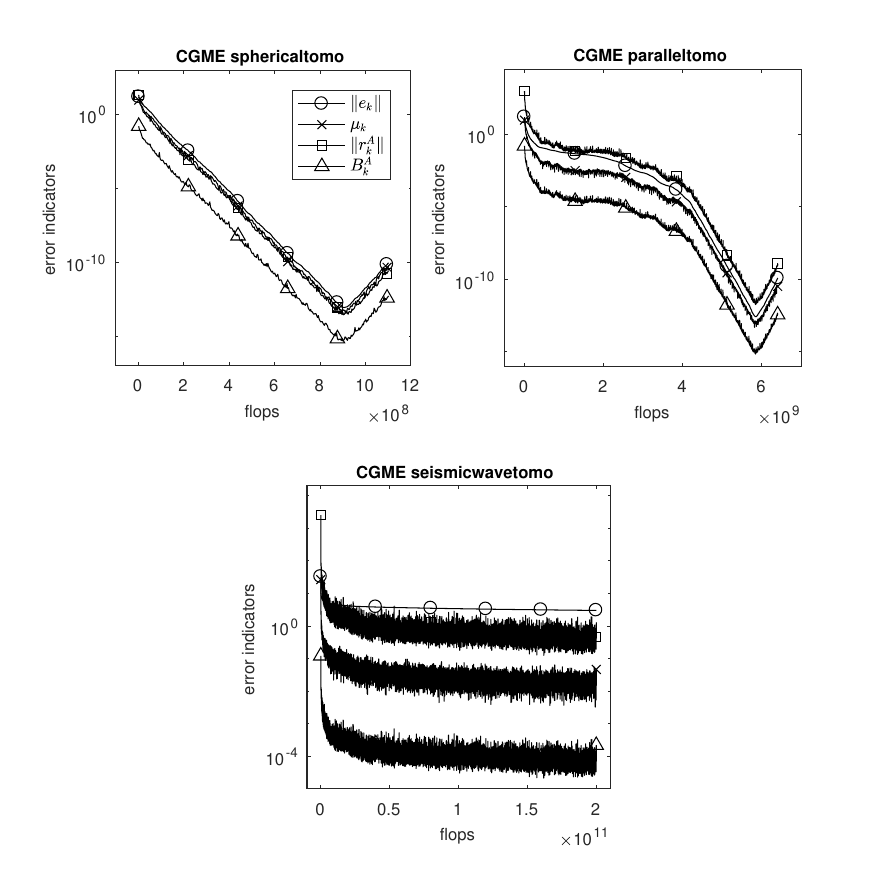}
\caption{Behavior of error indicators for CGME on model problems
with 64x64 pixels.}
\label{fig:craig}
\end{center}\end{figure}

In Figures \ref{fig:kaczmarz} and \ref{fig:craig} we plot the error
and several error indicators that may serve as a stopping criterion
for both BKME and CGME.
We plot the error and the step length
\[\|e_k\|:=\|x^*-x_k\|
\quad\text{and}\quad
\mu_k:=\|x_{k+1}-x_k\|
\]
of both methods, the residuals
\[r^A_k:=b-Ax_k
\quad\text{and}\quad
r^C_k:=g-Cx_k,\]
w.r.t.\ the linear systems $Ax=b$ and $Cx=g$, and the error indicators
\[B^A_k:=\frac{\|b-Ax_k\|}{\|A\|\|x_k\|+\|b\|}
\quad\text{and}\quad
B^C_k:=\frac{\|g-Cx_k\|}{\|C\|\|x_k\|+\|g\|}\]
motivated by \eqref{Bjorck:stopping:criterion} for both systems.
The indicator $B^A_k$ requires the residual $r^A_k$,
which is computed by CGME, but not by BKME.
We know, however, from Lemma \ref{TC} that $\|C\|\in(0,2)$, and data
gathered from the numerical examples in Table \ref{table:data:C} suggests
that $\|C\|\approx 1$ for tomography and related problems.
Since the residual $r^C_k$ is computed in line 3 of Algorithm \ref{Gram:Schmidt}
and $g=P(0)$, all ingredients of $B^C_k$ are known with a reasonable accuracy.

We find that all error indicators, including $B^A_k$ and $B^C_k$, oscillate
for both methods.
For BKME, the step length $\mu_k$ turns out to be the most reliable guess
for the true error,
while no general statement can be made about CGME.
The indicators $B^A_k$ and $B^C_k$ both seem to be much too optimistic to
be practically useful.

Figures \ref{fig:kaczmarz} and \ref{fig:craig} also show that both BKME
and CGME eventually become unstable in a very similar fashion, which means
that both methods require restarts when they do not make satisfactory progress.

%

\bibliographystyle{plain}
\bibliography{main}

\begin{thebibliography}{10}

\bibitem{Bauschke}
H.H. Bauschke, F.~Deutsch, H.~Hundal, and S.-H. Park.
\newblock Accelerating the convergence of the method of alternating
  projections.
\newblock {\em Trans. Amer. Math. Soc.}, 355(9):3433--3461, 2003.

\bibitem{Bjorck}
\AA. Bj\"{o}rck.
\newblock {\em Numerical methods in matrix computations}, volume~59 of {\em
  Texts in Applied Mathematics}.
\newblock Springer, Cham, 2015.

\bibitem{Censor}
Y.~Censor.
\newblock Row-action methods for huge and sparse systems and their
  applications.
\newblock {\em SIAM Rev.}, 23(4):444--466, 1981.

\bibitem{Elfving}
T.~Elfving.
\newblock Block-iterative methods for consistent and inconsistent linear
  equations.
\newblock {\em Numer. Math.}, 35(1):1--12, 1980.

\bibitem{Gearhart}
W.B. Gearhart and M.~Koshy.
\newblock Acceleration schemes for the method of alternating projections.
\newblock {\em J. Comput. Appl. Math.}, 26(3):235--249, 1989.

\bibitem{Gower}
R.M. Gower and P.~Richt\'{a}rik.
\newblock Randomized iterative methods for linear systems.
\newblock {\em SIAM J. Matr. Anal. Appl.}, 36(4):1660--1690, 2015.

\bibitem{Greenbaum}
A.~Greenbaum.
\newblock {\em Iterative methods for solving linear systems}, volume~17 of {\em
  Frontiers in Applied Mathematics}.
\newblock Society for Industrial and Applied Mathematics (SIAM), Philadelphia,
  PA, 1997.

\bibitem{Hansen}
P.C. Hansen and J.S. J{\o}rgensen.
\newblock A{IR} {T}ools {II}: algebraic iterative reconstruction methods,
  improved implementation.
\newblock {\em Numer. Algorithms}, 79(1):107--137, 2018.

\bibitem{Kaczmarz}
S.~Kaczmarz.
\newblock {A}ngenäherte {A}ufl\"{o}sung von {S}ystemen linearer {G}leichungen.
\newblock {\em Bulletin International de l’Académie Polonaise des Sciences
  et des Lettres. Classe des Sciences Mathématiques et Naturelles. Série A,
  Sciences Mathématiques}, 441:355--357, 1937.

\bibitem{Liesen}
J.~Liesen and Z.~Strakos.
\newblock {\em {Krylov Subspace Methods: Principles and Analysis}}.
\newblock Oxford University Press, 2012.

\bibitem{Meurant}
G.~Meurant and J.~Duintjer~Tebbens.
\newblock {\em Krylov methods for nonsymmetric linear systems---from theory to
  computations}, volume~57 of {\em Springer Series in Computational
  Mathematics}.
\newblock Springer, Cham, 2020.

\bibitem{Necoara}
I.~Necoara.
\newblock Faster randomized block {K}aczmarz algorithms.
\newblock {\em SIAM J. Matrix Anal. Appl.}, 40(4):1425--1452, 2019.

\bibitem{Needell}
D.~Needell and J.A. Tropp.
\newblock Paved with good intentions: analysis of a randomized block {K}aczmarz
  method.
\newblock {\em Linear Algebra Appl.}, 441, 2014.

\bibitem{Rieger}
J.~Rieger.
\newblock Generalized {G}earhart-{K}oshy acceleration for the {K}aczmarz
  method.
\newblock {\em Math. Comp.}, 92(341):1251--1272, 2023.

\bibitem{Saad}
Y.~Saad.
\newblock {\em Iterative methods for sparse linear systems}.
\newblock Society for Industrial and Applied Mathematics, Philadelphia, PA,
  second edition, 2003.

\bibitem{Steinerberger}
S.~Steinerberger.
\newblock A weighted randomized {K}aczmarz method for solving linear systems.
\newblock {\em Math. Comp.}, 90:2815--2826, 2021.

\bibitem{Strohmer}
T.~Strohmer and R.~Vershynin.
\newblock A randomized {K}aczmarz algorithm with exponential convergence.
\newblock {\em J. Fourier Anal. Appl.}, 15(2):262--278, 2009.

\bibitem{Tam}
M.K. Tam.
\newblock Gearhart-{K}oshy acceleration for affine subspaces.
\newblock {\em Oper. Res. Lett.}, 49(2):157--163, 2021.

\bibitem{Tanabe}
K.~Tanabe.
\newblock Projection method for solving a singular system of linear equations
  and its applications.
\newblock {\em Numer. Math}, 17:203--214, 1971.

\end{thebibliography}
\end{document}